\documentclass{aims}
\usepackage{amsmath}
  \usepackage{paralist}
  \usepackage{graphics} 
  \usepackage{epsfig} 
\usepackage{graphicx}  \usepackage{epstopdf}
 \usepackage[colorlinks=true]{hyperref}
\hypersetup{urlcolor=blue, citecolor=red}

\def\currentvolume{X}
 \def\currentissue{X}
  \def\currentyear{200X}
   \def\currentmonth{XX}
    \def\ppages{X--XX}

\newtheorem{theorem}{Theorem}[section]
\newtheorem{corollary}{Corollary}

\newtheorem{lemma}[theorem]{Lemma}
\newtheorem{proposition}{Proposition}

\theoremstyle{definition}
\newtheorem{definition}[theorem]{Definition}
\newtheorem{remark}{Remark}

\title[Fourier-Stieltjes in locally compact groups] 
      {Fourier-Stieltjes transform defined by induced representation on locally compact groups}

\author[Y. I. Akakpo,   K. Enakoutsa, K. Assiamoua  and M. N. Hounkonnou]{}


 \email{akjeremiji@yahoo.fr}
 \email{koffi.enakoutsa@csun.edu}
 \email{kofassiam@yahoo.fr}
  \email{norbert.hounkonnou@cipma.uac.bj, with copy to hounkonnou@yahoo.fr}

\begin{document}
\maketitle

\centerline{\scshape Yao Ihébami  Akakpo}
\medskip
{\footnotesize
 \centerline{1300 rue Ouimet App7, H4L 2L3 Montréal, Qc, Ca }
   \centerline{ICMPA-UNESCO Chair (UAC), 072 BP 50, Cotonou, Benin}
   \centerline{ Tel +15142294451}
} 

\medskip
\centerline{\scshape V.S. Koffi Assiamoua}
\medskip
{\footnotesize
 \centerline{ University of Lom\'e, Facult\'e Des Sience,,}
   \centerline{ Department of Mathematics (BP 1515 Lom\'e Togo),}
   \centerline{Tel +228 92 77 28 26}
}

\centerline{\scshape Koffi Enakoutsa}
\medskip
{\footnotesize
 \centerline{ California State University, Northridge, Department of Mathematics,}
   \centerline{18111 Nordhoff Street, Northridge, CA, USA}
}
\centerline{\scshape Mahouton Norbert Hounkonnou}
\medskip
{\footnotesize
 \centerline{ University of Abomey-Calavi,}
   \centerline{ International Chair in Mathematical Physics and Applications (ICMPA-UNESCO Chair),}
    \centerline{072 BP 50, Cotonou, Rep. of Benin}
   \centerline{Tel +229 95 06 26 89}
}

\bigskip


\begin{abstract}
In this work we extend the Fourier-Stieltjes transform of a vector measure and a continuous function defined on compact groups to locally compact groups. 
To do so, we consider a representation $L$ of a normal compact subgroup $K$ of a
locally compact group $G$, and we use a representation of $G$ induced by that of $L$. 
Then, we define the Fourier-Stieltjes transform of a vector measure and that of a
continuous function with compact support defined on $G$ from the representation of $G$. 
Then, we extend the Shur orthogonality relation established for compact groups to locally compact groups by
using the representations of $G$ induced by the unitary representations of one of its normal  compact subgroups. 
This extension enables us to develop a Fourier-Stieltjes transform in series form that is linear,
continuous, and invertible.
\end{abstract}

\section{Introduction}

\label{intro}
The vector measures generalizing scalar measures attracted a great interest in recent decades due to their numerous applications in functional analysis, control systems, signal analysis, quantum information, quantum theories, and many other domains of applications. 
For more details on vector measure theory, see for instance, 
%
%
\cite{Di,Bar} and \cite{assiamoua1} for some applications on compact groups. 
Also, Clarkson \cite{Ja} used theoretical ideas on vector measure to prove that many Banach spaces do not admit equivalent uniformly convex norms.
In the same vein,  Gel'fand \cite{Im} proved that $L _{1} \lbrack 0, 1\rbrack$ is not isomorphic to a dual of a Banach space. 
Lyapunov \cite{L} showed that the range of a (non-atomic) vector measure is closed and convex. 
Lyapunov \cite{L}' s work ccupies a prominent place in modern mathematics since it lies at the intersection of the theory of convex sets and measure theory. 
The Lyapunov convexity theorem became the starting point of numerous studies in the framework of mathematical analysis as well as in the realm of geometric research into the convex sets that are ranges of non-atomic vector measures  \cite{S}. 
In addition, Bartle \cite{Rg}, Dinculeanu, Kluv\'anek \cite{Di}, Dunford and Schwartz \cite{H-R},  and Lindenstrauss and Pelczy\'nki \cite{Jl} gave many seminal results on vector measure.
For instance, Diestel and Uhl Jr \cite{D-U} provided a comprehensive survey on vector measures. 
Applications of vector measures were discussed in 1980 in the work by Kluv\'anek \cite{Il}. 
Fern\'andez and Faranjo \cite{Af} studied the Rybakov's theorem for vector measures in Fr\'echet spaces.
Curbera and Ricker \cite{Gp} wrote a survey on vector measures, integration, and applications. 
More information on vector measures can also be found in \cite{D-U,Gp}. 
%
%
\\

To focus on our interest,
let $G$ be a locally compact group,   $m$  a vector measure on $G$ into a Banach algebra $\mathcal{A}$, $\lambda$ a left or right Haar measure on $G,$ and $f\in L_{1}(G,\lambda)$. 
If the group $G$ is abelian, the Fourier-Stieltjes transform of $m$ is given by the relation
\begin{eqnarray} 
 \hat{m}(\chi)  & = & \int_{G} \overline{\left<\chi,t\right>}dm(t),
 \end{eqnarray}
 while the Fourier transform of $f$ is given by 
 \begin{eqnarray}
\hat{f}(\chi) = \int_{G} \overline{\left<\chi,t\right>}f(t)d\lambda(t),
 \end{eqnarray}
where $\chi$ denotes a character of $G$. 
If $G$ is compact and $\mathcal{A} = \mathbb{C}$, then the Fourier-Stieltjes transform of $m$ is a family of endomorphisms $(\hat{m}(\sigma))_{\sigma\in \Sigma}$ given by
\begin{eqnarray} 
\left<\hat{m}(\sigma)  \xi, \eta \right> = \int_{G} \left< \overline{U}_{t}^{\sigma}\xi, \eta \right> d m(t),
\end{eqnarray}
 and the associated Fourier transform is provided by the relation
\begin{eqnarray}
\left<\hat{f}(\sigma)  \xi, \eta \right> = \int_{G} \left< \overline{U}_{t}^{\sigma}\xi, \eta \right>f(t) d \lambda(t),
\end{eqnarray}
where $U^{\sigma}$ denotes a unitary representation of the group $G$. 
If $G$ is compact and $\mathcal{A}$ is any Banach algebra, Assiamoua \cite{assiamoua1} defined a Fourier-Stieltjes transform of a bounded vector measure $m$ on $G$  as a family $(\hat{m}(\sigma))_{\sigma\in \Sigma}$ of sesquilinear mappings of $H_{\sigma} \times H_{\sigma}$ with values in $\mathcal{A}$ given by the relation,
\begin{eqnarray}	
\hat{m}(\sigma) ( \xi, \eta ) = \int_{G} \left< \overline{U}_{t}^{\sigma}\xi, \eta \right> d m(t),
\end{eqnarray}
and the  Fourier transform of a function $f \in L_{1} (G)$ as a family of continuous endomorphisms $(\hat{f}(\sigma))_{\sigma\in \Sigma}$ of sesquilinear applications of $H_{\sigma} \times H_{\sigma}$ with value in $\mathcal{A}$, given by the relation
\begin{eqnarray}
\hat{f}(\sigma) ( \xi, \eta ) = \int_{G} \left< \overline{U}_{t}^{\sigma}\xi, \eta \right>f(t) d \lambda(t).
\end{eqnarray}
In the continuation of previous investigations by \cite{assiamoua1},  the present work addresses a construction of the  Fourier-Stieltjes transform on locally compact groups from  a group representation induced by a representation of a compact subgroup.
For this purpose, we consider a locally compact group $G$, $K$ a compact normal subgroup of $G$, $\mu$ a $G$-invariant measure on the left coset space $G/K$, and $L^{\sigma}$ a unitary representation of $K$ into a separable Hilbert space $H_{\sigma}$. Then, we define the Fourier-Stieltjes transform of a vector measure on the locally compact group $G$ using the representation $U^{L^{\sigma}}$ of $G$ induced by $L^{\sigma}$.
In this context usually one uses Gel'fand transform to define Fourier transform.
We consider a locally compact group $G$ with Haar measure $dx$. Let $A$ be a unitary commutative Banach algebra and $X(A)$ its spectrum. 
The Gel'fand transform of $x$, $x\in A$, is the function $\mathcal{G}_{x}: X(A) \longrightarrow \mathbb{C}$ such that $\mathcal{G}_{x}(\chi) = \chi(x)$. 
The mapping $x \longmapsto \mathcal{G}_{x} : A \longrightarrow \mathbb{C}^{X(A)}$ is called the transformation of Gel'fand associated with $A$. 
The spherical Fourier transform is the Gel'fand transform associated with $L_{1}(G)^{\natural}$, the space of integrable, bi-invariant functions by a compact subgroup $K$ of $G$ on $G$. 
In this case, for $f\in L_{1}(G)^{\natural}$, $\mathcal{G}_{x}$ is denoted $\mathcal{F}f$ or $\hat{f}$ and is defined by 
\begin{eqnarray}
\hat{f}(\chi) & = & \int_{G}f(x)\chi (x^{-1})dx.
\end{eqnarray}
Our method has several advantages over the Gel'fand transform.
First, our Fourier transform is injective while the Gel'fand transform is not necessarly injective. 
Second the Gel'fand transform is limited to spherical functions only. In our case the Fourier transform exists for the whole $L_{1}(G,A)$, $1 \leq p < \infty$.
Finally, our method is constructive while in Gel'fand transformation the Fourier transformation is obtained by induction.
\\

The paper is organized as follows. In Section 1, we recall the definition of a vector measure and unit representation of a group which are useful in the sequel. 
In Section 2, we provide the proof of the Shur's orthogonality property in connection with induced representation and we define  the Fourier-Stieltjes transform of a vector measure and the Fourier transform of a function in $L_{1}(G, \mathcal{A})$. 

We also report in Section 2 several properties we discovered for our newly developed Fourier-Stieltjes transform.   
\section{Preliminaries}
\label{sec:1}
In this section, for the clarity of the development, we briefly recall useful  known main definitions and results, and set our notations.
We consider  a locally compact space $G,$ the  Banach spaces  $\mathcal{A}$ and $\mathcal{F}$  over the field $\mathbb{K}, \,  ( \mathbb{K} = \mathbb{R}$  or $\mathbb{C}),$
and denote by $\mathcal{K}(G, \mathcal{A})$ the vector space of all continuous functions $f : G \longrightarrow \mathcal{A}$  with a compact support, and by $\mathfrak{C}(G,\mathcal{A})$ the space of continuous functions $f : G \longrightarrow \mathcal{A}$.

For simplification, we write $\mathcal{K}(G)$ instead of $\mathcal{K}(G, \mathbb{R})$ or $\mathcal{K}(G, \mathbb{C}).$
For each subset $K$ of $G$, denote by 
$\mathcal{K}_{K}(G, \mathcal{A})$ the space of functions with support contained in $K$. $\mathcal{K}_{K}(G, \mathcal{A})$ is a subspace of $\mathcal{K}(G, \mathcal{A})$.
\begin{definition}  
 For every function $f \in \mathcal{K}(G, \mathcal{A})$, we define 
 \[\Vert f \Vert := \sup _{t\in G} \Vert f(t) \Vert_{\mathcal{A}}. \] 
The mapping $ f \mapsto \Vert f \Vert$ is a norm on each space  $\mathcal{K}_{K}(G, \mathcal{A}),$ and defines the topology of uniform convergence on $G$ over $\mathcal{K}(G)$. 
\end{definition}
\begin{definition}
 On $\mathcal{K}(G, \mathcal{A}),$  the topology of the  compact
convergence is the locally convex topology defined by the family of seminorms  \[\left\Vert f \right\Vert_{K} = \sup_{t \in K} \Vert f(t) \Vert_{\mathcal{A}}, \]
where $K$ takes the elements in the set of compact subsets of $G$.
\end{definition}
\begin{proposition}
The space $\mathcal{K}(G, \mathcal{A})$ is dense in the space $\mathfrak{C}(G,\mathcal{A})$ for the
topology of the compact convergence \cite{Di}.
\end{proposition}
\begin{definition} 
 A vector measure on $G$ with respect to two
spaces $\mathcal{A}$ and $\mathcal{F}$, or an $(\mathcal{A},\mathcal{F})$-measure on $G$, is any linear mapping 
$m : \mathcal{K}(G, \mathcal{A})\longrightarrow \mathcal{F} $ having the property that,  for each compact set $K \subset G,$ the restriction $m$ to the subspace $\mathcal{K}_{K}(G, \mathcal{A})$ is continuous for the topology of uniform convergence, i.e. for each compact set,  there exists a number $a_{K} > 0$ such that 
\[ \left\Vert m (f) \right\Vert \leq a_{K} \sup \left\lbrace \Vert f(t) \Vert_{\mathcal{A}} , \quad t \in K \right\rbrace. \]
\\

The value $m(f)$ of $m$ for a function $f \in \mathcal{K}(G, \mathcal{A})$ is called the
integral of $f$ with respect to $m$ also denoted by $\int_{G}f dm$ or $\int_{G}f(t) dm(t)$.
A vector measure is said to be dominated if there exists a positive measure $\mu$ such that 
\[\left\Vert \int_{G}f(t)dm(t) \right\Vert \leq \int_{G} \vert f(t) \vert d\mu(t), \quad    \quad f \in \mathcal{K}(G). \]
If $m$ is dominated, then there exists a smallest positive measure $\vert m \vert,$ called the modulus or the variation of $m$, that dominates it.  A positive measure is said to be bounded if it is continuous in the uniform norm topology of $\mathcal{K}(G).$
A vector measure is said to be bounded if it is dominated by a bounded positive measure. If $m$ is bounded,  then $\vert m \vert$ is also bounded. 

Denoting by $M_{1}(G,\mathcal{A})$ the Banach algebra of bounded vector measures on $G,$
 the mapping 
\begin{eqnarray} 
m \mapsto \Vert m \Vert = \int_{G} \chi _{G}d \vert m \vert
\end{eqnarray} 
  is a norm on $M_{1}(G,\mathcal{A})$, where $\chi _{G}$ represents the characteristic function of $G$.
\end{definition}
In the sequel, $K$ will denote a compact subgroup of $G$, $\nu$ and $\lambda$  left Haar measures  on $K$ and $G,$ respectively.
\begin{definition}\mbox{ }
Let $\mu$ be  a Radon measure on $G/K,$ the homogenous space of left $K$-cosets and $g$  an element of $G$. Define $\mu_{g}$ by $\mu_{g}(E) = \mu(gE)$ for Borel subsets $E$ of $G/K$. The measure $\mu$ is called $G-$invariant measure if $\mu_{g} = \mu$, for $g \in G$. See \cite{folland, gaal} for more details.
\end{definition}
%
%
Throughout the paper, $\mu$ will denote the $G-$ invariant measure on $G/K$.
\begin{theorem}\mbox{ }
 For any $ f \in \mathcal {K} (G), $ we have  \cite{folland,anahar}:
\begin{eqnarray} \label{inv}
\int_{G}f(g)d\lambda(g) &=& \int_{X}d\mu(\dot{g})\int_{K}f(gk)d\nu(k).
\end{eqnarray} 
The previous formula ( \ref{inv}) extends also to every $f\in L_{1}(G,\lambda, A )$ \cite{Bar2}.\\
\end{theorem}
\begin{definition}\mbox{ }
A unit representation of $G$ is a homomorphism $L$ from $G$ into the group $U (H)$ of the invertible unitary linear operators on some nonzero Hilbert space $H,$ which is continuous with respect to
the strong operator topology satisfying for $g_{1}, g_{2} \in G$,
\[L_{( g_{1} g_{2})}  =  L_{g_{1}}L_{g_{2}} \quad \text{and  } \quad L_{1}  =  Id_{H}.\]
$H$ is called the representation space of $L$, and its dimension is called the dimension or degree of $L$.
$\\$

Suppose $\mathcal{M}$ is a closed subspace of $H$. 
$\mathcal{M}$ is called an invariant subspace for $L$ if $L_{g}\mathcal{M} \subset \mathcal{M}$ $\forall g \in G$. If $\mathcal{M}$ is invariant and $\mathcal{M} \neq \{ 0\},$ then $L^{\mathcal{M}}$ such that
 \[L^{\mathcal{M}}_{g} = L_{g} \vert_{\mathcal{M}}\]
 defines a representation of $G$ on $\mathcal{M}$, called a subrepresentation of $L$. 
 If $L$ admits an invariant subspace that is nontrivial (i.e. $\neq \{0\}$ or $H$) then $L$ is called reducible, otherwise $L$ is irreducible. If $G$ is compact and $L$ irreducible then the dimension of $L$ is finite.
\end{definition} 
\begin{definition} \mbox{ }
Two unit irreductible representations $L$ and $V$ into
 $H$ and $N,$ respectively, are said to be equivalent if there is an isomorphism
$T :  H\longrightarrow N $ such that, $\forall t \in G,$
\[  T\circ L_{t} = V_{t}\circ T.\]
\end{definition}
Consider now the subgroup $K$, $\Sigma$ the coset space (called the dual object of $ K $), $\sigma\in \Sigma,$ $L^{\sigma}$ a representative of $\sigma$, $H_{\sigma}$ a representation space of $L^{\sigma},$ and $d_{\sigma}$ its dimension.
\begin{theorem}\mbox{ }
Let  $(L_{ij}^{\sigma})_{1 \leq i,j \leq d_{\sigma}}$ be the matrix of $L^{\sigma}$ in an orthonormal  basis  $(\xi_{i})_{i=1}^{d\sigma}$ of $H_{\sigma}.$ Then,  (see \cite{folland,anahar,men,B-R}), 
 \begin{eqnarray} \label{sch1}
 \int_{K}L_{ij}^{\sigma}(t)\overline{L_{lm}^{\sigma}}(t)d \nu (t) & = & \dfrac{\delta _{il} \delta_{jm}}{d_{\sigma} }
\end{eqnarray}  
and
 \begin{eqnarray} \label{sch2}
 \int_{K}L_{ij}^{\sigma}(t)\overline{L_{lm}^{\tau}}(t)d \nu (t) & = & 0 \text{ if } \sigma \neq \tau.
\end{eqnarray}
\end{theorem}  
Let  $q : G \longrightarrow G/K$ be the canonical quotient map of $G$ into $G/K$ and suppose $H_{\sigma}$ separable.
Denote by  $H^{L^{\sigma}}_{ 0}$ the set
 \begin{eqnarray}
H^{L^{\sigma}}_{ 0} = \left\lbrace u\in \mathfrak{C}(G,H_{\sigma})\mbox{ } : \mbox{ } q(\mbox{Supp(u))}\text{ is compact } \text{ and  } u(gk) = L^{\sigma}_{k^{-1}}u(g)\right\rbrace.
\end{eqnarray}
\begin{proposition}\mbox{ }
If $\eta :G \longrightarrow H_{\sigma}$ 
is  continuous  with compact support, then the function $u_{\eta}$ such that
\begin{eqnarray} \label{form}
u_{\eta}(g) = \int_{K}L^{\sigma}_{k}\eta(gk)d\nu(k)
\end{eqnarray} 
belongs to $H^{L^{\sigma}}_{ 0},$ and is uniformly continuous on $G$. Moreover, every 
element of $H^{L^{\sigma}}_{ 0}$ is of the form $u_{\eta}$. See \cite{folland} for more details.
\end{proposition}
\begin{proposition}\mbox{ }
The mapping: 
\begin{eqnarray}
(u,v)&\longmapsto & \left<u,v \right> = \int_{G/K}\left<u(g),v(g)\right>_{H_{\sigma}}d\mu(\dot{g}).
\end{eqnarray}
on  $H^{L^{\sigma}}_{0} \times H^{L^{\sigma}}_{0}$ is an inner product on $H^{L^{\sigma}}_{0}.$ 
\end{proposition}
$ G$ acts on $H^{L^{\sigma}}_{0}$ by left translation, $u \longmapsto L_{t}u$, so we obtain a unitary
representation of $G$ with respect to this inner product on $H^{L^{\sigma}}_{0}$.
The inner product is preserved by left translations, since $\mu$ is invariant. Hence, if we denote by $H^{L^{\sigma}}$ the Hilbert space completion of $H^{L^{\sigma}}_{0},$ the translation operators $L_{t}$  extend to unitary operators on $H^{L^{\sigma}}$. Then the map $t \longmapsto L_{t}u$ is continuous from $G$ to $H^{L^{\sigma}}$ for each $u \in H^{L^{\sigma}}_{0},$ and then, since the operators $L_{t}$ are uniformly bounded, they are strongly continuous on $H^{L^{\sigma}}.$ Hence, they define a unitary representation of G, called the representation induced by $L^{\sigma},$  denoted by $U^{L^{\sigma}}:$
\[ U_{t}^{L^{\sigma}}u(g) = L_{t}u(g) = u(t^{-1}g).\]
The representation space is denoted $H^{L^{\sigma}}$.
\begin{remark}\mbox{ }
\begin{itemize}
\item The representations of $G$ induced from $K$ are generally infinite-dimensional unless $G/K$ is a finite set.
\item
If induced representation $U^{L^{\sigma}}$ is irreducible then $L^{\sigma}$ is irreducible\cite{B-R}.\\
The converse of this statement is false.
\item $U^{L^{\sigma}} \in U(H^{L\sigma}),$ where  $U(H^{L\sigma})$  designates the group of  invertible unitary linear operators on $H^{L\sigma}$.
\end{itemize}
\end{remark} 
\textbf{Lebesgue Spaces}\\
Let $\lambda$ be a positve measure on the group $G$ into a Banach algebra $\mathcal{A}$.
For $f: G \longrightarrow \mathcal{A}$, put
\begin{eqnarray}
N_{p} (f) & = & \left( \int_{G}^{*}\left\Vert f(t) \right\Vert_{\mathcal{A}}^{p}d\lambda(t)\right)^{\dfrac{1}{p}}
\end{eqnarray}
$1 \leq p < \infty$, where  $\displaystyle\int_{G}^{*}$ designates the upper integral \cite{Bar} and 
\begin{eqnarray}
N_{\infty} (f) & = &  \inf \left\lbrace \alpha: \Vert f(t) \Vert _{\mathcal{A}} \leq \alpha \quad\lambda -\text{ almost everywhere }\right\rbrace. 
\end{eqnarray}
 $\mathcal{L}_{p}(G, \lambda, \mathcal{A})$  denotes the set of all  $\lambda$-measurable functions $f: G \longrightarrow \mathcal{A}$ such $N_{p}(f) < \infty$, $1\leq p \leq \infty$.
Generally the mapping $f \longmapsto N_{p}(f)$ is not a norm but a seminorm in $\mathcal{L}_{p}(G, \lambda, \mathcal{A})$,  $1\leq p \leq \infty$.

The relation $\mathcal{R}$ defined by $f \mathcal{R} g$ if and only if $f- g $ is null $\lambda$ almost everywhere is an equivalent relation in $\mathcal{L}_{p}(G, \lambda, \mathcal{A})$. Denoting the 
equivalence class of $f \in\mathcal{L}_{p}(G, \lambda, \mathcal{A})$ by $\lbrack f\rbrack$, it follows that $\Vert  \lbrack f\rbrack \Vert = N_{p}(f).$
\begin{definition}
The symbol $L_{p}(G, \lambda, \mathcal{A})$ will denote the set of equivalence classes $\lbrack f\rbrack$ of functions $f \in\mathcal{L}_{p}(G, \lambda, \mathcal{A})$.
\end{definition}
\begin{theorem}
The space $L_{p}(G, \lambda, \mathcal{A})$  is a normed linear space.
\end{theorem}
In the sequel the symbol $f$ rather than $ \lbrack f\rbrack $ will be used for an element in $L_{p}(G, \lambda, \mathcal{A})$.
\begin{theorem}
$N_{p} $ define a norm in $L_{p}(G, \lambda, \mathcal{A})$ noted $\Vert . \Vert_{p}$.
\end{theorem}
\begin{theorem} (H\"older inequality))
Consider $1 \leq p \leq \infty$, $q \geq 1$ such that $\dfrac{1}{p} + \dfrac{1}{q} = 1$ ( if $p = 1$, $q = \infty$) $f\in L_{p}(G, \lambda, \mathcal{A})$ and $ g \in L_{q}(G, \lambda, \mathcal{A})$  then the function $fg$ is $\lambda$-integrable ie $fg\in L_{1}(G, \lambda, \mathcal{A})$, and 
\begin{eqnarray}
\Vert fg \Vert_{1} & \leq & \Vert f \Vert _{p} \Vert g \Vert_{q}.
\end{eqnarray} 
\end{theorem}
\begin{theorem} (Minkowski inequality)
For $1 \leq p \leq \infty$ and $f, g \in  L_{p}(G, \lambda, \mathcal{A})$, we have 
\begin{eqnarray}
\Vert f + g \Vert _{p} & \leq & \Vert f \Vert_{p} + \Vert g \Vert_{p}.
\end{eqnarray}
\end{theorem}
\begin{remark}
We deduce by H\"older theorem that for $G$ such that $\lambda(G) < \infty $ and $1 \leq q \leq  p \leq \infty $, we have $ L_{\infty}(G, \lambda, \mathcal{A})\subset  L_{p}(G, \lambda, \mathcal{A})\subset  L_{q}(G, \lambda, \mathcal{A})\subset  L_{1}(G, \lambda, \mathcal{A})$ but if $\lambda(G) = \infty $ there are no inclusion between  $L_{p}(G, \lambda, \mathcal{A})$ and $L_{q}(G, \lambda, \mathcal{A})$ for $q \neq p$.
\end{remark}
\begin{proposition}
$L_{p}(G, \lambda, \mathcal{A})\cap L_{q}(G, \lambda, \mathcal{A})$ is dense in $L_{p}(G, \lambda, \mathcal{A})$ and in $L_{q}(G, \lambda, \mathcal{A})$.
\end{proposition}
\begin{remark}
If $m$ is a dominated vector measure into a Banach space $F$, then $\mathcal{L}_{p}(G, m, \mathcal{A})$  is by convention the space $\mathcal{L}_{p}(G,\vert m \vert, \mathcal{A}).$  
\end{remark}
\begin{theorem} \label{stone} (Stone-Weierstrass theorem for locally compact spaces)
Let $\Omega$ be a locally compact space, which is non-compact, let $ \mathfrak{C}_{0}^{\mathbb{K}}(\Omega)$ be the algebra of continuous $\mathbb{K}$-valued functions on $\Omega$ that vanish at infinity  equipped with the supremum norm and let $\mathcal{P} \subset \mathfrak{C}_{0}^{\mathbb{K}}(\Omega)$ with the following separation properties:
\begin{itemize}
\item[(i)] for any two points $\omega_{1}, \omega_{2} \in \Omega$ with $\omega_{1}\neq \omega_{2}$  there exists $f \in\mathcal{P}$ such that $f(\omega_{1}) \neq f(\omega_{2})$
\item[(ii)] for any $\omega \in \Omega$ there exists $f \in\mathcal{P}$ with $f(\omega) \neq0$.\\
\item[(a)] If $\mathbb{K} = \mathbb{R}$ then  $\mathcal{P}$ is dense in $\mathfrak{C}_{0}^{\mathbb{R}}(\Omega).$ 
\item[(b)] If $\mathbb{K} = \mathbb{C}$ and if $\mathcal{P}$ is a $_{*}$-subalgebra (i.e.  $f\in \mathcal{P}\Longrightarrow \overline{f}\in \mathcal{P}$) then  $\mathcal{P}$ is dense in $\mathfrak{C}_{0}(\Omega).$  
\end{itemize}
\end{theorem}
\section{Main results}
We have $K$ compact and normal, $H_{\sigma}$ is finite and separable;  then $H^{L^{\sigma}}$ is also separable.  Each of these spaces admits a Hilbertian basis according to Gram-Schmidt process.\\\\
In this work, we suppose that $K$ is chosen such that $L^{\sigma}$ and $U^{L^{\sigma}}$ are both irreducible.
\\

\subsection{Schur orthogonaltity relations } \mbox{  } \\
Let $(\theta_{i})_{i=1}^{\infty}$ be an orthonormal basis of $H^{L^{\sigma}}$ and  $ (\xi_{i})_{i=1}^{d\sigma}$ be an orthonormal  basis of $H_{\sigma}$.
We define
\begin{eqnarray} \label{trig}
u_{ij}^{L^{\sigma}}(t) &:=& \left< U_{t}^{L^{\sigma}}\theta_{j},\theta_{i}\right>_{H^{L^{\sigma}}}\cr
 &=& \int_{G/K}\left< \theta_{j}(t^{-1}g), \theta_{i}(g)\right>d\mu(\dot{g})
\end{eqnarray}
and 
\[L_{ij}^{\sigma}(k) := \left< L_{k}^{\sigma}\xi_{j},\xi_{i}\right>_{H_{\sigma}}
\]
As a result, there is a family of mappings $ (\alpha_{is})_{s=1}^{d\sigma}$ of $G$ into $\mathbb{K} with $ ($\mathbb{K} = \mathbb{R}$ or $ \mathbb{C}$) such that \[\theta _{i}(g) = \sum _{s = 1}^{d \sigma} \alpha_{is}(g)\xi _{s}.\]
We have: 
\[\delta_{ij} = \left< \theta _{j}, \theta_{i} \right>_{H^{L_{\sigma}}} = \int_{G/K} \left< \theta _{j} (g), \theta _{i} (g)\right >_{H_{\sigma}}d\mu (\dot{g}) = \sum _{ s = 1}^{d \sigma}\int_{G/K} \alpha_{js}(g)\overline{\alpha}_{is}(g) d\mu (\dot{g}).\]
\begin{proposition}\label{L2} \mbox{ }
$\forall\sigma \in \Sigma$ and $\forall i,j\in \{ 1, 2,...\}, $   $u_{ij}^{L^{\sigma}} \in \mathcal{K}(G).$ 
\end{proposition}
\begin{proof} 
The continuity of $u_{ij}^{L^{\sigma}}$ results from their constrction.\\ 
According to \ref{form}, for all $u\in H^{L^{\sigma}},$ thereis $\eta \in \mathcal{K}(G, H_{\sigma})$ such that $u_{\eta}(g) = \displaystyle\int_{K}L^{\sigma}_{k}\eta(gk)d\nu(k).$ It follows that $q(supp(u))\subset q( supp(\eta)).$
\\

Let $A$ be the support of $\eta.$ For every $g \in G$ and $k\in K,$ $gk\in A \Longrightarrow g \in Ak^{-1}.$\quad $AK = \displaystyle \bigcup _{k\in K}Ak^{-1}$ is compact because $A$ and $K$ are compact. F
Furthermore $\forall g\notin AK,$ $u(g) = 0$ then $supp(u) \subset AK.$ Since $supp(u)$ is a closed subset of $AK$ then it is compact.
\\

Now let $B$ and $C$ be $supp(\theta_{j})$ and $supp(\theta_{i}),$ respectively.
Assume $D = \left\lbrace t\in G : g\in C \Longrightarrow t^{-1}g \in B \right\rbrace$.
According to \ref{trig}, $\left< \theta_{j}(t^{-1}g), \theta_{i}(g)\right> = 0$ if $t \notin D,$ and then $u_{ij}^{L^{\sigma}}(t) = 0,$  then  $supp (u_{ij}^{L^{\sigma}}) \subset D.$ 
In addition, 
\begin{eqnarray*}
t^{-1}g \in B \text{ with } g\in C & \Longrightarrow & t \in CB^{-1}.
\end{eqnarray*}
Then $supp (u_{ij}^{L^{\sigma}}) \subset D \subset CB^{-1}.$ Since $CB^{-1}$ is compact, $supp (u_{ij}^{L^{\sigma}}) $ is also compact.
\end{proof}
Therefore, we have
\begin{corollary}\mbox{ }
\[\int_{G}\left\vert u_{ij}^{L^{\sigma}}(t)\overline{u}_{lm}^{L^{\sigma}}(t)\right\vert d\lambda(t) < \infty.\]
\end{corollary}
Since $K$ is normal, we have $ K\backslash G = G/K$ then 
\begin{lemma}
$\forall u \in H^{L^{\sigma}}$, $u(kt) = L_{k}u(t)$ and $u(tk) =  L_{k^{-1}}^{\sigma}u(t)$.
\end{lemma}
\begin{proof}
There is $\eta \in \mathfrak{C}(G, H_{\sigma})$ such that 
\[u(t) = \int_{K}L^{\sigma}(\xi)\eta(t\xi )d\nu(\xi ). \quad \text{\cite{folland}} \]
\begin{eqnarray*}
u(tk) & =  & \int_{K}L^{\sigma}(\xi)\eta(tk\xi)d\nu(\xi ), \quad k \in K\\
& = & \int_{K}L^{\sigma}(k^{-1}\xi)\eta(t\xi)d\nu(\xi ) \\
& = & L_{k^{-1}}^{\sigma}\int_{K}L^{\sigma}(\xi)\eta(t\xi)d\nu(\xi )\\
& = & L_{k^{-1}}^{\sigma}u(t).
\end{eqnarray*}
There exists $\beta \in \mathfrak{C}(G, H_{\sigma})$ such that 
\[u(t) = \int_{K}L^{\sigma}(\xi^{-1})\beta(\xi t )d\nu(\xi ).\quad \text{ \cite{B-R}} \]
\begin{eqnarray*}
u(kt) & =  & \int_{K}L^{\sigma}(\xi^{-1})\beta(\xi kt)d\nu(\xi )\\
& = & \int_{K}L^{\sigma}(k\xi^{-1})\beta(\xi t)d\nu(\xi ) \\
& = & L_{k}^{\sigma}\int_{K}L^{\sigma}(\xi^{-1})\beta(\xi t)d\nu(\xi )\\
& = & L_{k}^{\sigma}u(t). 
\end{eqnarray*}
\end{proof}
The following theorem shows the Schur orthogonality relation for the the case of the representation  $U^{L^{\sigma}}$ of $G$ induced by the unitary irreductible representation $L^{\sigma}$ of $K$.
\begin{theorem} \mbox{   }
We have
\begin{eqnarray}
\int_{G}u_{ij}^{L^{\sigma}}(t)\overline{u}_{lm}^{L^{\sigma}}(t)d\lambda(t) & = & \dfrac{c_{ijlm}}{d_{\sigma}}
\end{eqnarray}
where  \begin{eqnarray*}
 c_{ijlm} & : = &   \int _{G/K} d\mu(\dot{t})\sum_{r,s = 1}^{d_{\sigma}} \int_{G/K}\alpha_{js}(t^{-1}g)\overline{\alpha}_{ir}(g)d\mu(\dot{g})\int_{G/K}\alpha_{ms}(t^{-1}h)\overline{\alpha}_{lr}(h)d\mu(\dot{h})\\
 &  = &    \sum_{r,s = 1}^{d_{\sigma}} \int_{(G/K)^{3}} \alpha_{js}(t^{-1}g)\overline{\alpha}_{ir}(g) \alpha_{ms}(t^{-1}h)\overline{\alpha}_{lr}(h)d\mu(\dot{g})d\mu(\dot{h})d\mu(\dot{t})
 \end{eqnarray*}
and
\begin{eqnarray}
\int_{G}u_{ij}^{L^{\sigma}}(t)\overline{u}_{lm}^{L^{\tau}}(t)d\lambda(t) & = & 0 \text{ if } \sigma \neq \tau.
\end{eqnarray}
\end{theorem}
\begin{proof} By straightforward computation, we obtain:
\small{
\begin{eqnarray*}
 \int_{G/K} \left< \theta _{j} (gk), \theta _{i} (g) \right>_{H_{\sigma}}d\mu (\dot{g}) & = & \int_{G/K}  \left< L^{\sigma}(k^{-1})\theta _{j} (g), \theta _{i} (g) \right>_{H_{\sigma}}d\mu (\dot{g}) \\
  & = &\sum _{s, r = 1}^{d \sigma}\int_{G/K} \alpha_{js}(g)\overline{\alpha}_{ir}(g) \left< L^{\sigma}(k^{-1}) \xi_{s}, \xi_{r}\right>_{H_{\sigma}}  d\mu (\dot{g}) \\
  & = & \sum_{r,s = 1}^{d_{\sigma}} L_{rs}^{\sigma}(k^{-1}) \int_{G/K}\alpha_{js}(g)\overline{\alpha}_{ir}(g)d\mu(\dot{g}).
  \end{eqnarray*} 
Besides, we get:
 \begin{eqnarray*}    
u_{ij}^{L^{\sigma}}(tk) &=& \int_{G/K}\left< \theta_{j}(k^{-1}t^{-1}g), \theta_{i}(g)\right>d\mu(\dot{g}) \\
&=& \int_{G/K} \left< L^{\sigma}(k^{-1})\theta_{j}(t^{-1}g), \theta_{i}(g)\right>d\mu(\dot{g})\\
&=&  \int_{G/K} \left< L^{\sigma}(k^{-1})\theta_{j}(t^{-1}g), \theta_{i}(g)\right>d\mu(\dot{g}) \\
& = & \sum_{r,s = 1}^{d_{\sigma}} L_{rs}^{\sigma}(k^{-1}) \int_{G/K}\alpha_{js}(t^{-1}g)\overline{\alpha}_{ir}(g)d\mu(\dot{g})
\end{eqnarray*} 
and
 \begin{eqnarray*} 
\int_{G}u_{ij}^{L^{\sigma}}(t)\overline{u}_{lm}^{L^{\sigma}}(t)d\lambda(t) & = & \int_{G/K}\int_{K} u_{ij}^{L^{\sigma}}(tk)\overline{u}_{lm}^{L^{\sigma}}(tk)d\nu(k)d\mu(\dot{t}) \\
& = & \int _{G/K} d\mu(\dot{t})\int_{K}(\sum_{r,s = 1 }L_{rs}^{\sigma}(k^{-1})\int_{G/K} \alpha_{js}(t^{-1}g)\overline{\alpha_{ir}}(g)d\mu(\dot{g})  \times \\
&  &  \sum_{p,q = 1}L_{pq}^{\sigma}(k^{-1})\int_{G/K} \alpha_{mq}(t^{-1}h)\overline{\alpha_{lp}}(h)d\mu(\dot{h}))d\nu(k)\\ 
 & = & \int _{G/K} d\mu(\dot{t})\sum_{r,s,p,q = 1 }\int_{K}(L_{rs}^{\sigma}(k^{-1})\int_{G/K} \alpha_{js}(t^{-1}g)\overline{\alpha_{ir}}(g)d\mu(\dot{g}) \times \\ 
 &  & L_{pq}^{\sigma}(k^{-1})\int_{G/K} \alpha_{mq}(t^{-1}h)\overline{\alpha_{lp}}(h)d\mu(\dot{h}))d\nu(k). \\
\end{eqnarray*}
Hence, 
\begin{eqnarray*} 
\int_{G}u_{ij}^{L^{\sigma}}(t)\overline{u}_{lm}^{L^{\sigma}}(t)d\lambda(t) & = &\int _{G/K} d\mu(\dot{t})\sum_{r,s,p,q = 1 }\int_{K}(L_{rs}^{\sigma}(k^{-1}) L_{pq}^{\sigma}(k^{-1}) d\nu(k)\int_{G/K} \alpha_{js}(t^{-1}g)\overline{\alpha_{ir}}(g)d\mu(\dot{g}) \times \\  & & \int_{G/K} \alpha_{mq}(t^{-1}h)\overline{\alpha_{lp}}(h)d\mu(\dot{h}) \\
& = & \int _{G/K} d\mu(\dot{t})\sum_{r,s,p,q = 1 } \dfrac{\delta_{rp} \delta_{sq}}{d_{\sigma}}\int_{G/K} \alpha_{js}(t^{-1}g)\overline{\alpha_{ir}}(g)d\mu(\dot{g}) \times \\
& & \int_{G/K} \alpha_{mq}(t^{-1}h)\overline{\alpha_{lp}}(h)d\mu(\dot{h}) \text{ according to \ref{sch1} }\\ 
& = & \dfrac{1}{d_{\sigma}}\sum_{r,s = 1}^{d_{\sigma}} \int_{(G/K)^{3}}\alpha_{js}(t^{-1}g)\overline{\alpha}_{ir}(g) \alpha_{ms}(t^{-1}h)\overline{\alpha}_{lr}(h)d\mu(\dot{g})d\mu(\dot{h})d\mu(\dot{t}) \\
& = & \dfrac{c_{ijlm}}{d_{\sigma}}.
\end{eqnarray*}
If $\sigma \neq \tau$ then
\begin{eqnarray*} 
\int_{G}u_{ij}^{L^{\sigma}}(t)\overline{u}_{lm}^{L^{\tau}}(t)d\lambda(t) & = & \int _{G/K} d\mu(\dot{t})\sum_{r,s,p,q = 1 }\int_{K}(L_{rs}^{\sigma}(k^{-1}) L_{pq}^{\tau}(k^{-1}) d\nu(k)\int_{G/K} \alpha_{js}(t^{-1}g)\overline{\alpha_{ir}}(g)d\mu(\dot{g}) \times \\  & & \int_{G/K} \alpha_{mq}(t^{-1}h)\overline{\alpha_{lp}}(h)d\mu(\dot{h}) = 0 \text{ according to \ref{sch2} } 
\end{eqnarray*}
}
\end{proof}
In the case of a particular orthonormal basis, the orthogonality relation reduces to the following:
\begin{corollary} \mbox{ }
Choosing  an orthonormal  basis $(\xi_{i})_{i=1}^{d\sigma}$  of $H_{\sigma}$ such that 
\begin{eqnarray}\label{normal}
\int_{(G/K)^{3}}\alpha_{js}(t^{-1}g)\overline{\alpha}_{ir}(g) \alpha_{ms}(t^{-1}h)\overline{\alpha}_{lr}(h)d\mu(\dot{g})d\mu(\dot{h})d\mu(\dot{t}) = \left\{\begin{array}{cl}
 \dfrac{1}{d_{\sigma}^{2}}  &\text{  if } j = m \text{  and } i = l \\
  & \\
0 &\text{ if not}
\end{array}\right. 
\end{eqnarray}
(or, also, $c_{ijlm} = \delta_{il}\delta_{jm}$)
leads to 
\begin{eqnarray}
\int_{G}u_{ij}^{L^{\sigma}}(t)\overline{u}_{lm}^{L^{\sigma}}(t)d\lambda(t) & = & \dfrac{\delta_{il}\delta_{jm}}{d_{\sigma} }.
\end{eqnarray}
\end{corollary}
\begin{proof} 
With the conditions \ref{normal}, we have 
\begin{eqnarray*}
\int_{G}u_{ij}^{L^{\sigma}}(t)\overline{u}_{lm}^{L^{\sigma}}(t)d\lambda(t) & = & \small{\dfrac{1}{d_{\sigma}} \sum_{r,s = 1}^{d_{\sigma}} \int_{(G/K)^{3}}\alpha_{js}(t^{-1}g)\overline{\alpha}_{ir}(g) \alpha_{ms}(t^{-1}h)\overline{\alpha}_{lr}(h)d\mu(\dot{g})d\mu(\dot{h})d\mu(\dot{t})}\\
& = &  \dfrac{\delta_{il}\delta_{jm}}{d_{\sigma}}\sum_{i,j = 1}^{d_{\sigma}}\dfrac{1}{d_{\sigma}^{2}}\\
& = &  \dfrac{\delta_{il}\delta_{jm}}{d_{\sigma}}\dfrac{d_{\sigma}^{2}}{d_{\sigma}^{2}}\\
\int_{G}u_{ij}^{L^{\sigma}}(t)\overline{u}_{lm}^{L^{\sigma}}(t)d\lambda(t) & = & \dfrac{\delta_{il}\delta_{jm}}{d_{\sigma} }. 
\end{eqnarray*}
In the sequel, $(\theta_{i})_{i=1}^{\infty}$ will designate an orthonormal basis of $H^{L^{\sigma}}$ and  $ (\xi_{i})_{i=1}^{d\sigma}$ that of $H_{\sigma},$ where $ (\xi_{i})_{i=1}^{d\sigma}$ is chosen such that
\begin{eqnarray*}
\int_{G/K}d\mu(\dot{t})\int_{G/K}\alpha_{js}(t^{-1}g)\overline{\alpha}_{ir}(g)d\mu(\dot{g})\int_{G/K}\alpha_{ms}(t^{-1}h)\overline{\alpha}_{lr}(h)d\mu(\dot{h}) = \left\{\begin{array}{cl}
 \dfrac{1}{d_{\sigma}^{2}}  &\text{  if } j = m \text{  and } i = l \\
  & \\
0 &\text{ if not.}
\end{array}\right. 
\end{eqnarray*}
\end{proof}
\subsection{Fourier-Stieltjes transform on a locally compact group with a compact subgroup} 
\begin{definition}  
Assume $m\in M_{1}(G,\mathcal{A})$.
We shall define the Fourier-Stieljes transform of an arbitrary measure $m$ as the family $(\hat{m}(\sigma))_{\sigma\in \Sigma}$ of sesquilinear mappings on $H^{L^{\sigma}} \times H^{L^{\sigma}}$ into  $\mathcal{A}$, given by the relation
\begin{eqnarray}
\hat{m}(\sigma)(u,v) =  \int_{G}\left< \overline{U}_{t}^{L^{\sigma}}u,v \right>_{H^{L^{\sigma}}}dm(t) = \int_{G}\int_{G/K} \left<\overline{u}(t^{-1}g),v(g)\right>d\mu(\dot{g})dm(t).
\end{eqnarray}
The Fourier transform  of a function $f\in L_{1}(G,\mathcal{A},\lambda)$, where $\lambda$ denotes a Haar measure on $ G $, is a family $(\hat{f}(\sigma))_{\sigma\in \Sigma}$ of sesquilinear mappings of $H^{L^{\sigma}} \times H^{L^{\sigma}}$  into $\mathcal{A},$ given by the relation
\begin{eqnarray}
\hat{f}(\sigma)(u,v) =  \int_{G}\left< \overline{U}_{t}^{L^{\sigma}}u,v \right>_{H^{L^{\sigma}}}f(t)d\lambda(t) = \int_{G}\int_{G/K} \left<\overline{u}(t^{-1}g),v(g)\right>f(t)d\mu(\dot{g})d\lambda(t).
\end{eqnarray}
\end{definition}
\subsection{Properties of the Fourier-Stieljes transform}
\mbox{  }\\
Let denote by $\mathcal{S}(\Sigma,\mathcal{A}) = \displaystyle{\prod_{\sigma\in\Sigma}}\mathcal{S}(H^{L^{\sigma}}\times H^{L^{\sigma}}, \mathcal{A}),$ where $\mathcal{S}(H^{L^{\sigma}}\times H^{L^{\sigma}}, \mathcal{A})$ is the set of all the sesquilinear mappings of $H^{L^{\sigma}}\times H^{L^{\sigma}}$ into $\mathcal{A}$.
Also, assume $ \displaystyle{\prod_{\sigma\in\Sigma}}\mathcal{S}(H^{L^{\sigma}}\times H^{L^{\sigma}}, \mathcal{A})$ is a vector space for addition and multiplication by a scalar of mappings. 
\begin{definition}
For $\Phi\in  \mathcal{S}(\Sigma,\mathcal{A}) = \displaystyle{\prod_{\sigma\in\Sigma}}\mathcal{S}(H^{L^{\sigma}}\times H^{L^{\sigma}}, \mathcal{A})$, note $\Vert \Phi \Vert_{\infty}$ the quantity is defined by 
\begin{eqnarray}
\Vert \Phi \Vert_{\infty} & = &  \sup\lbrace \Vert \Phi (\sigma)\Vert : \sigma\in\Sigma\rbrace
\end{eqnarray}
with $\Vert\Phi(\sigma)\Vert = \sup \lbrace  \Vert \Phi(\sigma)(u,v)\Vert _{\mathcal{A}}\mbox{ }\mbox{  } : \mbox{  }\Vert u \Vert_{H^{L^{\sigma}}} \leq 1,\mbox{  }\Vert v \Vert_{H^{L^{\sigma}}} \leq 1\rbrace$.
Let define:
\begin{itemize}
\item[(i)] $\displaystyle \mathcal{S}_{\infty}(\Sigma,\mathcal{A})$ =  $\left\lbrace \Phi \in \mathcal{S}(\Sigma,\mathcal{A}): \Vert \Phi \Vert_{\infty} < \infty 
\right\rbrace$.  
\item[(ii)] $\displaystyle \mathcal{S}_{00}(\Sigma,\mathcal{A})$ = $\left\lbrace \Phi \in \mathcal{S}(\Sigma,\mathcal{A}): \lbrace  \sigma \in\Sigma :\Phi(\sigma) \neq 0 \rbrace \quad \mbox{is finite}\right\rbrace$.
\item[(iii)] $\displaystyle\mathcal{S}_{0}(\Sigma,\mathcal{A})$ = $\left\lbrace \Phi \in \mathcal{S}(\Sigma,\mathcal{A}): \forall \varepsilon > 0 \quad\lbrace  \sigma \in\Sigma  \mbox{ }\mbox{ }:\mbox{ }\mbox{ }\Vert\Phi(\sigma)\Vert > \varepsilon \rbrace \quad\mbox{is finite}\right\rbrace$.
\end{itemize}
\end{definition}   
\begin{proposition}\label{dec} \mbox{  } 
Let us consider $\Phi$ in  $\mathcal{S}(\Sigma,\mathcal{A}).$ 
There is the matrix $(a_{ij}^{\sigma})_{1 \leq i, j \leq \infty}$,  $a_{ij}^{\sigma}\in \mathcal{A}$ such that 
\[\Phi (\sigma) = \sum _{i,j = 1}^{\infty}d_{\sigma}  a_{ij}^{\sigma}\hat{u}_{ij}^{L^{\sigma}} \] 
with $\hat{u}_{ij}^{L^{\sigma}}$ the Fourier transform of $u_{ij}^{L^{\sigma}}$ given by
\[\hat{u}_{ij}^{L^{\sigma}}(\sigma)(u,v)  =  \int_{G}\left< \overline{U}_{t}^{L^{\sigma}}u,v \right>_{H^{L^{\sigma}}}u_{ij}^{L^{\sigma}}(t)d\lambda(t).\]
\end{proposition}
\begin{proof} 
For all $u, v \in H^{L^{\sigma}}$  such that
$u = \displaystyle \sum _{j = 1}^{\infty} \beta_{j} \theta_{j}$ and  $v = \displaystyle \sum _{i = 1}^{\infty} \gamma_{i} \theta_{i}$,
\[ \Phi (\sigma) (u, v) = \sum _{i,j = 1}^{\infty}\beta_{j}\overline{\gamma}_{i}\Phi (\sigma) (\theta_{j}, \theta_{i}).\]
Defining $(\Phi (\sigma) (\theta_{j}, \theta_{i}))$ = $( a_{ij}^{\sigma})$, the matrix of $\Phi (\sigma) $ in the basis $(\theta_{i})_{i=1}^{\infty}$ and
we have
\begin{eqnarray*}
\hat{u}_{ij}^{L^{\sigma}}(\sigma)(u,v) & = & \int_{G}\left< \overline{U}_{t}^{L^{\sigma}}u,v \right>_{H^{L^{\sigma}}}u_{ij}^{L^{\sigma}}(t)d\lambda(t)\\
& = & \sum_{l,m = 1}^{\infty} \beta_{m}\overline{\gamma}_{l}\int_{G} \overline{u}_{lm}^{L^{\sigma}}u_{ij}^{L^{\sigma}}d\lambda (t) =  \dfrac{\beta_{j}\overline{\gamma}_{i}}{d_{\sigma}}
\end{eqnarray*}.
\begin{eqnarray*}
\Phi (\sigma) (u, v) & = & \sum _{i,j = 1}^{\infty} a_{ij}^{\sigma}\beta_{j}\overline{\gamma}_{i}  =   \sum _{i,j = 1}^{\infty}d_{\sigma}  a_{ij}^{\sigma}\hat{u}_{ij}^{L^{\sigma}}(u,v)
\end{eqnarray*}
which yields \[\Phi (\sigma) = \sum _{i,j = 1}^{\infty}d_{\sigma} a_{ij}^{\sigma}\hat{u}_{ij}^{L^{\sigma}}.\]
\end{proof}
\begin{corollary}
For any $f \in L_{1}(G, \mathcal{A}, \lambda)$, there is a matrix  $(a_{ij})_{1 \leq i, j \leq \infty}$, $a_{ij}\in \mathcal{A}$  such that  
\[\hat{f} (\sigma) = \sum _{i,j = 1}^{\infty}d_{\sigma} a_{ij}^{\sigma}\hat{u}_{ij}^{L^{\sigma}}.  \]
\end{corollary}
\begin{lemma}\label{stone2}
Let us denote by $\mathcal{L}_{\sigma}(G)$ the set of finite linear combinations functions $t  \longmapsto \left<U_{t}^{L^{\sigma}}u,v \right>_{H^{L^{\sigma}}}$ and $\displaystyle\mathcal{L}(G) = \displaystyle\bigcup _{\sigma \in \Sigma}\mathcal{L}_{\sigma}(G).$
$\mathcal{L}(G)$ is dense in $\mathfrak{C}_{0}(G)$ for the topology of the uniform convergence, where $\mathfrak{C}_{0}(G)$ denotes the  space of continuous funtions of $G$ into $\mathbb{K}$ vanishing at infinity.
\end{lemma}
\begin{proof}
\begin{itemize}
\item[(i)] First we have $ \mathcal{L}(G) \subset \mathfrak{C}_{0}(G)$ moreover  for $f \in \mathcal{L}(G)$ we have abviously $ \overline{f} \in \mathcal{L}(G).$
\item[(ii)] Let $t_{1}$ and $t_{2}$ be two elements of $ G$ such that  $t_{1} \neq t_{2}$. Let us suppose by contradiction that for every $f\in \mathcal{L}(G),$ $ f(t_{1}) =  f(t_{2})$. Thus, by chosing $f = f_{k} = u_{ik}^{L^{\sigma}}$  we have $\left < U_{t_{1}}^{L^{\sigma}}\theta_{i}, \theta_{k} \right> = \left < U_{t_{2}}^{L^{\sigma}}\theta_{i}, \theta_{k} \right>$  for $i,k \in \{ 1, 2,...\}$. Let us fix $i$ and make $k$ running over the set $\{ 1, 2,...\}.$ Then, antilinear mappings $\phi_{U_{t_{1}}^{L^{\sigma}}\theta_{i}} = \left < U_{t_{1}}^{L^{\sigma}}\theta_{i}, . \right> $ and $\phi_{U_{t_{2}}^{L^{\sigma}}\theta_{i}} = \left < U_{t_{2}}^{L^{\sigma}}\theta_{i}, . \right> $ are identically equal in $ H^{L^{\sigma}}$ therefore $U_{t_{1}}^{L^{\sigma}}\theta_{i} = U_{t_{2}}^{L^{\sigma}}\theta_{i} \Longrightarrow \theta_{i} = U_{t_{2}t_{1}^{-1}}^{L^{\sigma}}\theta_{i} \Longrightarrow U_{t_{2}t_{1}^{-1}}^{L^{\sigma}} \equiv I_{ H^{L^{\sigma}}}.$
It means that $t_{2} = t_{1}$. There is a contradiction; hence, there is $f\in \mathcal{L}(G)$ such that $ f(t_{1}) \neq  f(t_{2})$.
\item[(iii)] Let $t$ be any element of $G$; since $U_{t}^{L^{\sigma}}$ is invertible, there exist $i,k \in \{ 1, 2,...\}$ such that $U_{t}^{L^{\sigma}}\theta_{i} = \theta_{k}$. Therefore, we have $\left < U_{t}^{L^{\sigma}}\theta_{i}, \theta_{k} \right> = \left < \theta_{k}, \theta_{k} \right > = 1 \neq 0$ i.e $f(t) \neq 0$.
\end{itemize}
According to theorem \ref{stone} $ \mathcal{L}(G)$ is dense in $\mathfrak{C}_{0}(G)$.
\end{proof}
\begin{theorem}\label{inj} The mapping $ m \longmapsto \hat{m}$ from $ M_{1}(G,\mathcal{A})$ into $\mathcal{S}_{\infty}(\Sigma,\mathcal{A})$ is linear,  injective and continuous.
\end{theorem}
\begin{proof}  Let $m, n \in M_{1}(G,\mathcal{A})$ such that $\hat{m} = \hat{n}.$ For any $u, v \in H^{L^{\sigma}} \times H^{L^{\sigma}}$ and   any $\sigma \in \Sigma,$ we have:
\begin{eqnarray}\hat{m} = \hat{n}&\Longleftrightarrow &\int_{G}\left< \overline{U_{t}^{L^{\sigma}}}u,v \right>_{H^{L^{\sigma}}}dn(t) =  \int_{G}\left< \overline{U_{t}^{L^{\sigma}}}u,v \right>_{H^{L^{\sigma}}}dm(t)\cr
&\Longleftrightarrow& \displaystyle\int_{G}\left< \overline{U_{t}^{L^{\sigma}}}u,v \right>_{H^{L^{\sigma}}}d(n-m)(t) = 0
\end{eqnarray}
for any $\sigma$ in $\Sigma,$ and $u$,$v$ in $H^{L^{\sigma}}$.

$ \mathcal{L}(G)$ is dense  in  $\mathfrak{C}_{0}(G)$ according to the previous lemma. Then, $\mathcal{L}(G)$ is dense in $\mathcal{K}(G)$. Thus $n-m$ can be viewed as a linear map of $\mathcal{K}(G)$ which is identically null. Then $\displaystyle\int_{G}\left< \overline{U_{t}^{L^{\sigma}}}u,v \right>_{H^{L^{\sigma}}}d(n-m)(t) = 0$  $\Longrightarrow  n- m \equiv 0$ ie $m = n.$ The mapping $ m \longmapsto \hat{m}$ is therefore injective.

Next, let us now prove the continuity of the mapping $ m \longmapsto \hat{m}$.   
\begin{eqnarray*}
\Vert \hat{m}( \sigma) \Vert  & = & \sup \left\lbrace  \Vert  \hat{m}(\sigma) (u,v) \Vert_{\mathcal{A}} :  \Vert u \Vert _{H^{L^{\sigma}}},  \leq 1 \text{ and  }\Vert v \Vert _{H^{L^{\sigma}}},  \leq 1 \right\rbrace \\
& = & \sup \left\lbrace  \left\Vert  \int_{G}\left< \overline{U_{t}^{L^{\sigma}}}u,v \right>_{H^{L^{\sigma}}}dm(t)\right\Vert_{\mathcal{A}} :  \Vert u \Vert _{H^{L^{\sigma}}},  \leq 1 \mbox{ and  }\Vert v \Vert _{H^{L^{\sigma}}},  \leq 1 \right\rbrace \\
& \leq &  \int_{G} \chi _{G}d \vert m \vert =  \Vert m \Vert,
\end{eqnarray*}
since $ U_{t}^{L^{\sigma}}$ is unitary.
Thus, $\Vert \hat{m}( \sigma) \Vert \leq \Vert m \Vert $, $\sigma \in \Sigma$ and $\Vert \hat{m} \Vert_{\infty} \leq \Vert m \Vert $. As a consequence,  $ \hat{m} \in  \mathcal{S}_{\infty}(\Sigma,\mathcal{A})$ and the mapping is continuous.

\begin{corollary}\label{inj2}
The mapping $f \longmapsto \hat{f}$ from $ L_{1}(G, \lambda, \mathcal{A})$ into $\mathfrak{S}_{\infty}(\Sigma,\mathcal{A})$ is linear, injective, and continuous.
\end{corollary}  
\textbf{Proof}
Here we consider $f,g \in L_{1}(G, \lambda, \mathcal{A})$ such that $\hat{f} = \hat{g}$ then 
\begin{eqnarray*} 
\int_{G}\left< \overline{U_{t}^{L^{\sigma}}}u,v \right>_{H^{L^{\sigma}}}fd\lambda(t)  =\int_{G}\left< \overline{U_{t}^{L^{\sigma}}}u,v \right>_{H^{L^{\sigma}}}fd\lambda(t) &\Longrightarrow & \int_{G}\left< \overline{U_{t}^{L^{\sigma}}}u,v \right>_{H^{L^{\sigma}}}(f - g)d\lambda(t) = 0.\\
\end{eqnarray*}
For the same reasons as the previous proof we have $f-g \equiv 0$ then $f = g$.
\begin{eqnarray*}
\Vert \hat{f}( \sigma) \Vert  & = & \sup \left\lbrace  \Vert  \hat{f}(\sigma) (u,v) \Vert_{\mathcal{A}} :  \Vert u \Vert _{H^{L^{\sigma}}},  \leq 1 \text{ and  }\Vert v \Vert _{H^{L^{\sigma}}},  \leq 1 \right\rbrace \\ \\
& = & \sup \left\lbrace  \left\Vert  \int_{G}\left< \overline{U_{t}^{L^{\sigma}}}u,v \right>_{H^{L^{\sigma}}}f(t)d\lambda(t)\right\Vert :  \Vert u \Vert _{H^{L^{\sigma}}},  \leq 1 \mbox{ and  }\Vert v \Vert _{H^{L^{\sigma}}},  \leq 1 \right\rbrace \\ \\
& \leq & \int_{G}\Vert f(t) \Vert d\lambda(t) = \Vert f \Vert_{1}.
\end{eqnarray*}
Thus, $\Vert \hat{f}( \sigma) \Vert \leq \Vert f \Vert_{1} $, $\sigma \in \Sigma$ and $\Vert \hat{f} \Vert_{\infty} \leq \Vert f \Vert_{1} $. Hence,  $ \hat{f} \in  \mathfrak{S}_{\infty}(\Sigma,\mathcal{A})$ and the mapping is continuous.
\end{proof}
\subsection{Modified Peter-Weyl theorem} 
The following theorem gives the inverse formula of Fourier transform in the context of our work. 
\begin{theorem}
For every $f\in L_{2}(G,\mathcal{A})$, there is $a_{ij}^{\sigma} \in \mathcal{A},$  $ 1 \leq i, j < \infty$, $\sigma \in \Sigma$ such that 
\begin{eqnarray}
f & = & \displaystyle\sum_{\sigma \in \Sigma}d_{\sigma}\displaystyle\sum_{i,j = 1}^{\infty}a_{ij}^{\sigma}u_{ij}^{L^{\sigma}}.
\end{eqnarray} 
\end{theorem}
\begin{proof}
With Proposition \ref{dec} for $g \in L_{1}(G, \lambda, \mathcal{A}),$ there exists $(a_{ij}^{\sigma})_{1 \leq i, j \leq \infty}$,  $a_{ij}^{\sigma}\in \mathcal{A}$ such that 
\[\hat{g} (\sigma) = \sum _{i,j = 1}^{\infty}d_{\sigma}  a_{ij}^{\sigma}\hat{u}_{ij}^{L^{\sigma}} \] 
with $\hat{u}_{ij}^{L^{\sigma}}(\sigma)(u,v)  = \displaystyle \int_{G}\left< \overline{U}_{t}^{L^{\sigma}}u,v \right>_{H^{L^{\sigma}}}u_{ij}^{L^{\sigma}}(t)d\lambda(t)$ and 
$\hat{g} (\sigma) (\theta_{j}, \theta_{i})$ = $ a_{ij}^{\sigma}.$\\ Thus $\hat{g}  = \displaystyle \sum _{\sigma\in \Sigma}d_{\sigma} \sum _{i,j = 1}^{\infty} a_{ij}^{\sigma}\hat{u}_{ij}^{L^{\sigma}}$.
\\

It is well known that $L_{1}(G, \lambda, \mathcal{A}) \cap L_{2}(G, \lambda, \mathcal{A})$ is dense in $L_{2}(G, \lambda, \mathcal{A})$.
Then, for every $f \in L_{2}(G, \lambda, \mathcal{A})$ there exists a sequence $(f_{n})$ in $L_{1}(G, \lambda, \mathcal{A}) \cap L_{2}(G, \lambda, \mathcal{A})$
such \\that  $f_{n}\xrightarrow[n\to\infty]{}  f $ in $L_{2}(G, \lambda, \mathcal{A})$.
$\hat{f_{n}} = \displaystyle \sum _{\sigma\in \Sigma}d_{\sigma} \displaystyle \sum _{i,j = 1}^{\infty}\mbox{  }^{n} a_{ij}^{\sigma}\hat{u}_{ij}^{L^{\sigma}} $ where
$^{n} a_{ij}^{\sigma} = \hat{f_{n}} (\sigma) (\theta_{j}, \theta_{i}).$ We have \\ $\hat{f_{n}}(\sigma) (\theta_{j}, \theta_{i})\xrightarrow[n\to\infty]{}  \hat{f}(\sigma) (\theta_{j}, \theta_{i}) $ ie there  exists $a_{ij}^{\sigma}$ in $\mathcal{A}$ such that  $^{n} a_{ij}^{\sigma} \xrightarrow[n\to\infty]{}  a_{ij}^{\sigma}.$\\ Finally  $\hat{f}  = \displaystyle \sum _{\sigma\in \Sigma}d_{\sigma} \sum _{i,j = 1}^{\infty} a_{ij}^{\sigma}\hat{u}_{ij}^{L^{\sigma}}.$

According to the corollary \ref{inj2} the mapping $f \longmapsto \hat{f}$ is injective; then 
$\hat{f}  = \displaystyle \sum _{\sigma\in \Sigma}d_{\sigma} \sum _{i,j = 1}^{\infty} a_{ij}^{\sigma}\hat{u}_{ij}^{L^{\sigma}}$ i.e. $\hat{f}  = \displaystyle \widehat{\left(\displaystyle \sum _{\sigma\in \Sigma}d_{\sigma} \sum _{i,j = 1}^{\infty} a_{ij}^{\sigma}u_{ij}^{L^{\sigma}}\right)}$   $\Longrightarrow  f   = \displaystyle \sum _{\sigma\in \Sigma}d_{\sigma} \sum _{i,j = 1}^{\infty} a_{ij}^{\sigma}u_{ij}^{L^{\sigma}}.$
\end{proof}
\begin{corollary}
For every $f \in L_{2}(G,\mathcal{A})$, we have 
\begin{eqnarray}
f  & = & \displaystyle\sum_{\sigma \in \Sigma}d_{\sigma}\displaystyle\sum_{i,j = 1}^{\infty}\hat{f}(\theta_{j},\theta_{i}) u_{ij}^{L^{\sigma}}. 
\end{eqnarray}
\end{corollary}
\section{Concluding remarks}
 In this work, we have developed a formalism for the Fourier-Stieljes transform on a locally compact group. More specifically, we  have:
\begin{itemize}
\item constructed the Fourier-Stieljes transform  of a bounded vector measure on a locally compact group into a Banach algebra,  
\item derived the Fourier transform of an integrable function, and
 \item  discussed corresponding relevant  properties including linearilty, injectivity, and continuity.  
\end{itemize}
Also, using the induced representation and its consequences, 
we established the associated Shur's orthogonality relation on a locally compact group. 

\medskip
Received xxxx 20xx; revised xxxx 20xx.
\medskip

\end{document}